\documentclass[a4paper,12pt]{article}
\usepackage[margin=1in]{geometry}  

\usepackage{graphicx}              
\usepackage{amsmath}
\usepackage{amscd}               
\usepackage{amsfonts} 
\usepackage{amssymb}             
\usepackage{amsthm}                

\newtheorem{thm}{Theorem}[section]
\newtheorem{defn}[thm]{Definition}
\newtheorem{lem}[thm]{Lemma}

\newtheorem{remark}[thm]{Remark}

\newcommand{\RR}{\mathbb{R}}      

\newcommand{\lfl}{\left\lfloor }  
\newcommand{\rfl}{\right\rfloor} 

\begin{document}

\title{A uniform estimate for rough paths}

\author{Terry J. Lyons\protect\footnote{Mathematical and Oxford-Man Institutes, University of Oxford, 24-29 St. Giles, Oxford, OX1 3LB. Email: terry.lyons@oxford-man.ox.ac.uk. } \and Weijun Xu\protect\footnote{Mathematical and Oxford-Man Institutes, University of Oxford, 24-29 St. Giles, Oxford, OX1 3LB. Email: xu@maths.ox.ac.uk. }}

\date{}

\maketitle

\abstract{It is well known that for two $p$-rough paths, if their first $\lfl p \rfl$ levels of interated integrals are close in $p$-variation sense, then all levels of their iterated integrals are close. In this paper, we prove that a similar result holds for the paths provided the first $\lfl p \rfl$ terms are close in a 'uniform' sense. The estimate is explicit, dimension free, and only involves the $p$-variation of two paths and the 'uniform' distance between the first $\lfl p \rfl$ terms. Applications include estimation of the difference of the signatures of two uniformly close paths (\cite{Lyons and Xu}), and convergence rates for Gaussian rough paths (\cite{Riedel and Xu}). 
}

\section{Introduction}

\subsection{Motivation}

The classical continuity theorem in rough paths (Theorem 2.2.2 in \cite{Lyons 1998}) states that if $X$ and $Y$ are two $p$-rough paths whose $p$-variation are both controlled by $\omega$ and such that 
\begin{align} \label{classical continuity condition}
\left\| X_{s,t}^{k} - Y_{s,t}^{k}  \right\| \leq \epsilon \frac{\omega(s,t)^{\frac{k}{p}}}{\beta (\frac{k}{p})!}, \qquad \forall k = 1, \cdots, \lfl p \rfl, 
\end{align}
then \eqref{classical continuity condition} holds for all $k \geq 1$. The proof is by an induction argument, which depends on the value of the exponent on the control, namely $\frac{k}{p}$. Although it is powerful in many places, there are certain problems for which we need a more convenient (and weaker) assumption. More precisely, we assume
\begin{align} \label{the general assumption}
\left\| X_{s,t}^{k} - Y_{s,t}^{k} \right\| < \epsilon \frac{\omega(s,t)^{\frac{k - \delta}{p}}}{\beta (\frac{k}{p})!}, \qquad \forall k = 1, \cdots, \lfl p \rfl. 
\end{align}
where $\delta \in [0,1]$. We wish to study whether similar estimates hold for $k \geq \lfl p \rfl + 1$. It is easy to see that the classical assumption \eqref{classical continuity condition} corresponds to $\delta = 0$. 

Such estimates are useful in a number of problems. For example, consider the following two linear differential equations
\begin{align} \label{the first equation}
dx_{t} = Ax_{t} d \gamma_{t}, \qquad x_{0} = a, 
\end{align}
and
\begin{align} \label{the second equation}
dy_{t} = Ax_{t} d \tilde{\gamma}_{t}, \qquad y_{0} = a, 
\end{align}
where $\gamma, \tilde{\gamma}: [0,1] \rightarrow \RR^{d}$ are two paths of bounded variations whose lengths are both controlled by $\omega$. Suppose further that
\begin{align} \label{uniformly close}
\sup_{t \in [0,1]} |\gamma_{t} - \tilde{\gamma}_{t}| < \epsilon, 
\end{align}
and one wishes to estimate the difference of the solution flow $|x_{t} - y_{t}|$. This question involves estimating the differences between all higher degrees of iterated integrals of $\gamma$ and $\tilde{\gamma}$, which are called signatures (we will give a precise definition in the next section). If we let $X$ and $Y$ to be the signatures of $\gamma$ and $\tilde{\gamma}$, then assumption \eqref{uniformly close} can be written as
\begin{align*}
\left\| X_{s,t}^{1} - Y_{s,t}^{1}  \right\| \leq 2 \epsilon = 2 \epsilon \omega(s,t)^{0}.  
\end{align*}
We see that it falls in the assumption \eqref{the general assumption} with $p = 1$ and $\delta = 1$. We will come back to this question at the end of this paper. 

Our estimates also apply to obtaining the convergence rates of Gaussian rough paths. For details, we refer to the recent works \cite{Friz and Riedel Gaussian rough paths} and \cite{Riedel and Xu}. 

\bigskip

\begin{flushleft}
\textit{Notation}. In what follows, $p$ will always be a number that is at least $1$. We use $\lfl p \rfl$ to denote the largest integer that does not exceed $p$, and let $\{p\} = p - \lfl p \rfl$ to be the fracal part of $p$. 
\end{flushleft}

\bigskip

\subsection{Main results}

Before stating our main result, let us explain briefly why the induction argument for the classical continuity theorem does not work directly here. As mentioned earlier, the induction argument depends on the exponent $\frac{n}{p}$. More precisely, the exponent for the level $n + 1 = \lfl p \rfl + 1$ is expected to be
\begin{align} \label{classical condition on the exponent}
\frac{\lfl p \rfl + 1}{p} > 1. 
\end{align}
This ensures that when one repeats Young's trick of dropping points, the total sum will converge. However, this condition is not satisfied in our problem \eqref{the general assumption} unless $\delta < 1 - \{p\}$. 

To this point, one may wonder whether one can immediately get the estimate by raising the control to an appropriate power so that the new control satisfies assumption \eqref{classical continuity condition}. Unfortunately this does not work, for there is no fixed power that one can do it in a homogeneous way for all $k \leq \lfl p \rfl$. Furthermore, the new control will in general fail to be superadditive. 

The idea is that we will compute one more term by hand, namely the level $\lfl p \rfl + 1$. After obtaining the estimate for this term, the exponent on the control will satisfy condition \eqref{classical condition on the exponent}, and we can use the usual induction argument for higher levels. 

In computing the estimate for the level $\lfl p \rfl + 1$, we again use Young's trick of dropping points. As mentioned earlier, the control in the 'uniform' distance assumption does not satisfy \eqref{classical condition on the exponent}, so in general the sum will not converge. Our idea is to move part of the exponent of $\epsilon$ to fill in the gap in the exponent in the control so that it reaches the necessary level for the sum to converge. Technicallly, it involves combining the 'uniform' distance together with the $p$-variation control for both paths, whose exponent for the level $\lfl p \rfl + 1$ is strictly greater than $1$. But since the latter does not involve the distance $\epsilon$, there is a trade-off between the convergence of the sum and the exponent for $\epsilon$ in higher degrees. In general, if $\delta > 1 - \{p\}$, then the exponent on $\epsilon$ for higher levels will be strictly less than one. We will make this precise in the proof of the main theorem. 

Below is the main theorem of this paper.

\bigskip

\begin{thm} \label{main theorem}
Let $p > 1$, and $\{p\} = p - \lfl p \rfl$ be the fractional part of $p$. Let $X,Y$ be two multiplicative functionals with finite $p$-variation which are both controlled by $\omega$ with the same constant $\beta$. That is to say: $\forall s,t \in [0,1]$ and $\forall k = 1, \cdots, \lfl p \rfl$, we have
\begin{align} \label{condition for the control}
\left\|  X_{s,t}^{k} \right\| \leq \frac{\omega(s,t)^{\frac{k}{p}}}{\beta (\frac{k}{p})!}, \qquad \left\|  Y_{s,t}^{k} \right\| \leq \frac{\omega(s,t)^{\frac{k}{p}}}{\beta (\frac{k}{p})!}. 
\end{align}
Suppose further that there exists an $\epsilon < 1$ such that $\forall k = 1, \cdots, \lfl p \rfl$, we have
\begin{align} \label{uniform distance assumption}
\left\|  X_{s,t}^{k} - Y_{s,t}^{k}  \right\| < \epsilon \cdot \frac{\omega(s,t)^{\frac{k - \delta}{p}}}{\beta (\frac{k}{p})!}, 
\end{align}
where $\delta \in [0,1]$. Then, the followings hold for all $(s,t) \in \Delta$: 
\begin{enumerate}
\item If $\delta \in [0, 1 - \{p\})$, and $\beta$ satisfies
\begin{align} \label{first condition}
\beta > p / [1 - (\frac{1}{2})^{\frac{1 - \{p\} - \delta}{p}}], 
\end{align}
then for all $k \geq 1$, we have
\begin{align} \label{first situation}
\left\| X_{s,t}^{k} - Y_{s,t}^{k} \right\| \leq \epsilon \cdot \frac{\omega(s,t)^{\frac{k-\delta}{p}}}{\beta (\frac{k}{p})!}. 
\end{align}

\item If $\delta = 1 - \{p\}$, and $\beta$ satisfies 
\begin{align} \label{second condition}
\beta > \frac{4p \cdot 2^{\frac{1-\{p\}}{p}}}{1 - (\frac{1}{2})^{\frac{1-\{p\}}{p}}}, 
\end{align}
then for all $k \geq \lfl p \rfl + 1$, we have
\begin{align} \label{second situation}
\left\| X_{s,t}^{k} - Y_{s,t}^{k}  \right\| < \epsilon \bigg(1 + \frac{p}{1-\{p\}} + \log_{2}\frac{\omega(0,1)}{\epsilon^{(1-\{p\})/p}} \bigg) \frac{\omega(s,t)^{\frac{k - 1 + \{p\}}{p}}}{\beta (\frac{k}{p})!}. 
\end{align}

\item If $p$ is non-integer, $\delta \in (1 - \{p\}, 1]$, and $\beta$ satisfies
\begin{align} \label{third condition}
\beta > 2p \bigg[ \frac{2^{(2p + \delta)/p}}{1 - (\frac{1}{2})^{(\delta - 1 + \{p\}) / p}} + \frac{1}{1 - (\frac{1}{2})^{(1 - \{p\})/p}} \bigg], 
\end{align}
then for all $k \geq \lfl p \rfl + 1$, we have
\begin{align} \label{third situation}
\left\| X_{s,t}^{k} - Y_{s,t}^{k}  \right\| < \epsilon^{\frac{1 - \{p\}}{\delta}} \cdot \frac{\omega(s,t)^{\frac{k - 1 + \{p\}}{p}}}{\beta (\frac{k}{p})!}. 
\end{align}

\end{enumerate}
\end{thm}

\bigskip

We see that the assumption becomes weaker as $\delta$ increases from $0$ to $1$. If $\delta < 1 - \{p\}$, then there is no essential differnce with the classical theorem, and the rate $\epsilon$ is maintained for all higher levels. On the other hand, if $\delta > 1 - \{p\}$, then we have a lost in the power of $\epsilon$ in higher levels. At the borderline case where $\delta = 1 -\{p\}$, we get a logorithmic correction.

\bigskip

\begin{remark}
The classical continuity theorem corresponds to the case $\delta = 0$. The case $p = 1$ and $\delta = 1$ was studied in \cite{Lyons and Xu}, and it was shown that the logarithmic correction can be removed in this situation. But the method there only works for $p = 1$, and does not generalize to $p > 1$. 
\end{remark}

\bigskip

\subsection{Structure of the paper}

This paper is organized as follows. In section 2, we introduce the concepts and notations from rough path theory that are necessary for our current problem. Sections 3 and 4 are devoted to the proof of the main theorem. In section 5, we come back to the example mentioned in the motivation, and explain how our estimates apply to this problem.

\bigskip

\textbf{Acknowledgements.} Both authors wish to acknowledge the support of the Oxford-Man Institute. The research of Terry Lyons is supported by EPSRC grants EP/F029678/1 and EP/H000100/1, and the European Research Council under the European Union's Seventh Framework Programme (FP7-IDEAS-ERC) / ERC grant agreement no. 291244. Weijun Xu thanks Peter Friz and Sebastian Riedel for helpful discussions during his visit to Berlin.

\bigskip

\section{Elements from rough path theory}

In this section, we introduce some concepts and notations from rough paths that are necessary for the current paper. For a detailed account of rough path theory, we refer the reader to \cite{Friz and Victoir book} and \cite{Lyons and Qian}. 

Fix the time interval $[0,1]$. For any $0 \leq s < t \leq 1$, write $\Delta_{s,t} = \{(u_{1}, u_{2}) | s < u_{1} < u_{2} < t\}$. In case $(s,t) = (0,1)$, we will simply write $\Delta = \Delta_{0,1}$. For every integer $N$, write
\begin{align*}
T^{N}(\RR^{d}) = \RR \oplus \RR^{d} \oplus \cdots \oplus (\RR^{d})^{\otimes N}. 
\end{align*}
If $\gamma: [0,1] \rightarrow \RR^{d}$ is a path of bounded variation, then the signature of $\gamma$ is defined by
\begin{align*}
X_{s,t}(\gamma) = 1 + X_{s,t}^{1} + \cdots + X_{s,t}^{n} + \cdots, 
\end{align*}
where
\begin{align*}
X_{s,t}^{n} = \int_{s < u_{1} < \cdots < u_{n} < t} d\gamma_{u_{1}} \otimes \cdots d\gamma_{u_{n}} \in T^{n}(\RR^{d}). 
\end{align*}
It is well known that $X$ is a multiplicative functional, in the sense that for any $s < u < t$, we have
\begin{align*}
X_{s,u} \otimes X_{u,t} = X_{s,t}. 
\end{align*}
If we restrict to the truncated tensor $(1, X_{s,t}^{1}, \cdots, X_{s,t}^{n})$, then it is a multiplicative functional in $T^{n}$.

\bigskip

\begin{defn}
A function $\omega: \Delta \rightarrow \RR^{+}$ is a control if it is continuous in both entries, vanishes on the diagonal and superadditive in the sense that for all $s < u < t$, we have
\begin{align*}
\omega(s,u) + \omega(u,t) \leq \omega(s,t). 
\end{align*}
\end{defn}

\bigskip

\begin{defn}
Let $X: \Delta \rightarrow T^{n}$ be a multiplicative functional. We say $X$ has finite $p$-variation controlled by $\omega$ with a constant $\beta$ if for all $0 < s < t < 1$ and all $k = 1, \cdots, n$, we have
\begin{align*}
\left\| X_{s,t}^{k} \right\| \leq \frac{\omega(s,t)^{\frac{k}{p}}}{\beta (\frac{k}{p})!}. 
\end{align*}
\end{defn}

\bigskip

A $p$-rough path is a multiplicative functional in $T^{\lfl p \rfl}$ with finite $p$-variation controlled by some $\omega$ with a constant $\beta$. Given a multiplicative functional in $T^{n}$, it is natural to ask whether it has a multiplicative extension to $T^{m}$ for $m > n$. The following (Theorem 2.2.1. in \cite{Lyons 1998}) answers it in affirmative provided it has finite $p$-variation controlled by some $\omega$.

\bigskip

\begin{thm} \label{Lyons extension theorem}
Suppose $X: \Delta \rightarrow T^{n}$ is a multiplicative functional in $T^{n}$ with finite $p$-variation controlled by $\omega$ with constant $\beta$, where $n \geq \lfl p \rfl$ and
\begin{align*}
\beta \geq \frac{p}{1 - (\frac{1}{2})^{\frac{\lfl p \rfl + 1}{p}-1}}, 
\end{align*}
then for any $m > n$, $X$ has a unique multiplicative extension to $T^{m}$ with finite $p$-variation controlled by $\omega$ with the same $\beta$. 
\end{thm}

\bigskip

\section{Some preliminary lemmas}

We first introduce some notations. Let $X_{s,t} = (1, X_{s,t}^{1}, \cdots, X_{s,t}^{n}) \in T^{(n)}$ be a multiplicative functional with finite $p$-variation controlled by $\omega$, where $n \geq \lfl p \rfl$. Define
\begin{align*}
\hat{X}_{s,t} = (1, X_{s,t}^{1}, \cdots, X_{s,t}^{n}, 0) \in T^{(n+1)}. 
\end{align*}
For any partition $\mathcal{P} = \{s = u_{0} < u_{1} < \cdots < u_{N-1} < u_{N} = t\}$, define
\begin{align*}
\hat{X}_{s,t}^{\mathcal{P}}:= \hat{X}_{s,u_{1}} \otimes \cdots \otimes \hat{X}_{u_{N-1}, t} \in T^{(n+1)}. 
\end{align*}
The following lemma gives a construction of the unique multiplicative extension of $X$ to higher degrees. It was proved in Theorem 2.2.1. in \cite{Lyons 1998}.

\bigskip

\begin{lem} \label{construction of multiplicative extension}
Let $X = (1, X_{s,t}^{1}, \cdots, X_{s,t}^{n})$ be a multiplicative functional in $T^{(n)}$. Let $\mathcal{P} = \{s = u_{0} < \cdots < u_{N} = t\}$ be any parition of $(s,t)$, and $\mathcal{P}^{j}$ be the partition of $(s,t)$ obtained by removing $u_{j}$ from $\mathcal{P}$. Then, 
\begin{align*}
\hat{X}_{s,t}^{\mathcal{P}} - \hat{X}_{s,t}^{\mathcal{P}^{j}} = (0, \cdots, 0, \sum_{k=1}^{n} X_{u_{j-1}, u_{j}}^{k} \otimes X_{u_{j}, u_{j+1}}^{n+1-k}) \in T^{(n+1)}. 
\end{align*}
Suppose further that $X$ has finite $p$-variation controlled by $\omega$, and $n \geq \lfl p \rfl$. Then, the limit
\begin{align*}
\lim_{\left\| \mathcal{P}  \right\| \rightarrow 0} \hat{X}_{s,t}^{\mathcal{P}} \in T^{(n+1)}
\end{align*}
exists. Furthermore, it is the unique multiplicative extension of $X$ to $T^{(n+1)}$, and is also controlled by $\omega$. 
\end{lem}

\bigskip

\begin{remark}
In the proof of the classical extension and continuity theorem (Theorems 2.2.1 and 2.2.2 in \cite{Lyons 1998}), one constructs multiplicative functionals of higher levels from those in lower levels in the sense of Lemma \ref{construction of multiplicative extension}, since the limit exists as long as the mesh of the partitions tends to $0$. On the other hand, given that the limit in Lemma \ref{construction of multiplicative extension} exists, we can choose a particular sequence of partitions which is convenient for calculation, and yields an improved constant $\beta$. 
\end{remark}

\bigskip

\begin{defn}
A partition $\mathcal{P}$ is a $K$-dyadic partition of $(s,t)$ with respect to the control $\omega$ if
\begin{align*}
\mathcal{P} = \{s = u_{0} < u_{1} < \cdots < u_{2^{K}-1} < u_{2^{K}} = t \}, 
\end{align*}
and for all $j = 1, 3, 5, \cdots, 2^{K}-1$, we have
\begin{align*}
\omega (u_{j-1}, u_{j}) = \omega (u_{j}, u_{j+1}). 
\end{align*}
\end{defn}

\bigskip

\begin{defn}
A partition $\mathcal{P}$ is a total $K$-dyadic partition of $[s,t]$ with respect to the control $\omega$ if for each $m = 0, 1, 2, \cdots, K$, the subpartition
\begin{align*}
\mathcal{P}_{K-m} = \{s = u_{0} < u_{2^{m}} < u_{2^{m} \cdot 2} < \cdots < u_{2^{m} \cdot 2^{K-m}} = t\}
\end{align*}
is a $(K-m)$-dyadic partition of $[s,t]$ with respect to $\omega$. 
\end{defn}

\bigskip

\begin{lem}
Let $\omega: \Delta \rightarrow \RR^{+}$ be a control. For any interval $[s,t]$, for each integer $K$, there exists a unique total $K$-dyadic partition $\mathcal{P}_{K}$ of $[s,t]$ with respect to $\omega$. Furthermore, $\mathcal{P}_{K+1}$ can be obtained from $\mathcal{P}_{K}$ by inserting one single point between every two consecutive points in $\mathcal{P}_{K}$ in a unique manner. 
\end{lem}
\begin{proof}
Since $\omega$ is continuous and strictly monotone in both variables, there exists a unique point $u \in (s,t)$ such that
\begin{align*}
\omega(s,u) = \omega (u,t). 
\end{align*}
Thus, $\mathcal{P}_{1} = \{s < u < t\}$ is the unique total $1$-dyadic partition. Suppose 
\begin{align*}
\mathcal{P}_{K} = \{s = u_{0} < u_{1} < \cdots < u_{2^{K}-1} < u_{2^{K}} = t  \}
\end{align*}
is the unique total $K$-dyadic partition of $[s,t]$ with respect to $\omega$. Then, for every $u_{j} < u_{j+1} \in \mathcal{P}_{K}$, there exists a unique point $v_{j+1} \in (u_{j}, u_{j+1})$ such that
\begin{align*}
\omega (u_{j}, v_{j+1}) = \omega (v_{j+1}, u_{j+1}). 
\end{align*}
Thus, 
\begin{align*}
\mathcal{P}_{K+1} = \{ s = u_{0} < v_{1} < u_{1} < \cdots < u_{2^{K} - 1} < v_{2^{K}} < u_{2^{K}} = t\}
\end{align*}
is the desired unique total dyadic-$K$ partition of $(s,t)$ with respect to $\omega$. 
\end{proof}

\bigskip

The next lemma is crucial for our estimates. It was first proved in \cite{Lyons 1998} with a constant $\frac{1}{p^{2}}$ on the left hand side. Recently, Hara and Hino improved it to $\frac{1}{p}$ in \cite{Hara and Hino}.

\bigskip

\begin{lem}
(Neo-classical inequality) Let $p \geq 1$, and define $x!:= \Gamma(x+1)$. Then for any $x,y \in \RR$, we have
\begin{align*}
\frac{1}{p}\sum_{k=0}^{n}\frac{x^{\frac{k}{p}}}{(\frac{k}{p})!} \cdot \frac{y^{\frac{n-k}{p}}}{(\frac{n-k}{p})!} \leq \frac{(x+y)^{\frac{n}{p}}}{(\frac{n}{p})!}. 
\end{align*}
\end{lem}

\bigskip

\begin{lem} \label{main lemma}
Let $X,Y$ be two multiplicative functionals with finite $p$-variation both controlled by $\omega$ with the same constant $\beta$. Suppose further that there is an $\epsilon < 1$ such that
\begin{align*}
\left\| X_{s,t}^{k} - Y_{s,t}^{k} \right\| < \epsilon \cdot \frac{\omega(s,t)^{\frac{k - \delta}{p}}}{\beta (\frac{k}{p})!}
\end{align*}
for each $k = 1, 2, \cdots, n$, where $n \geq \lfl p \rfl$, and $\delta \in [0,1]$. Then, we have
\begin{align*}
&\phantom{1}\left\|  (\hat{X}_{s,t}^{\mathcal{P}_{K+1}} - \hat{Y}_{s,t}^{\mathcal{P}_{K+1}})^{n+1} \right\| - \left\| (\hat{X}_{s,t}^{\mathcal{P}_{K}} - \hat{Y}_{s,t}^{\mathcal{P}_{K}})^{n+1} \right\| \\
&\leq \min \bigg\{  \frac{\epsilon p}{\beta^{2} (\frac{n+1}{p})!} \cdot 2^{\frac{2p+\delta}{p}} \big(\frac{1}{2^{K}} \big)^{\frac{n+1-p-\delta}{p}} \omega(s,t)^{\frac{n+1-\delta}{p}}, \phantom{11}  \big( \frac{1}{2^{K}} \big)^{\frac{n+1}{p}-1} \frac{2p \omega(s,t)^{\frac{n+1}{p}}}{\beta^{2} (\frac{n+1}{p})!} \bigg\}.  
\end{align*}
for all $K = 0, 1, 2, \cdots$. 
\end{lem}
\begin{proof}
For any partition $\mathcal{P} = \{s = u_{0} < \cdots < u_{j} < \cdots < u_{N} = t\}$, let $\mathcal{P}^{j}$ be the partition with the point $u_{j}$ removed from $\mathcal{P}$. By Lemma \ref{construction of multiplicative extension}, we have
\begin{align*}
&\phantom{11}(\hat{Y}_{s,t}^{\mathcal{P}} - \hat{Y}_{s,t}^{\mathcal{P}^{j}})^{n+1} - (\hat{X}_{s,t}^{\mathcal{P}} - \hat{X}_{s,t}^{\mathcal{P}^{j}})^{n+1} \\
&= \sum_{k=1}^{n}(R_{u_{j-1}, u_{j}}^{k} \otimes X_{u_{j}, u_{j+1}}^{n+1-k} + X_{u_{j-1}, u_{j}}^{k} \otimes R_{u_{j}, u_{j+1}}^{n+1-k} + R_{u_{j-1}, u_{j}}^{k} \otimes R_{u_{j}, u_{j+1}}^{n+1-k}), 
\end{align*}
where
\begin{align*}
R_{s,t} = Y_{s,t} - X_{s,t}. 
\end{align*}
Thus, we have
\begin{align*}
\left\|  (\hat{X}_{s,t}^{\mathcal{P}} - \hat{Y}_{s,t}^{\mathcal{P}})^{n+1}  \right\|  &=  \left\|  (\hat{X}_{s,t}^{\mathcal{P}^{j}} - \hat{Y}_{s,t}^{\mathcal{P}^{j}})^{n+1}  +  (\hat{X}_{s,t}^{\mathcal{P}} - \hat{X}_{s,t}^{\mathcal{P}^{j}})^{n+1}  -  (\hat{Y}_{s,t}^{\mathcal{P}} - \hat{Y}_{s,t}^{\mathcal{P}^{j}})^{n+1}     \right\| \\
&\leq \left\|  (\hat{X}_{s,t}^{\mathcal{P}^{j}} - \hat{Y}_{s,t}^{\mathcal{P}^{j}})^{n+1}  \right\|  +  \sum_{k=1}^{n} \bigg(    \left\|  R_{u_{j-1}, u_{j}}^{k} \otimes X_{u_{j}, u_{j+1}}^{n+1-k}  \right\|  \\
& + \left\|  X_{u_{j-1}, u_{j}}^{k} \otimes R_{u_{j}, u_{j+1}}^{n+1-k}  \right\|  +  \left\|  R_{u_{j-1}, u_{j}}^{k} \otimes R_{u_{j}, u_{j+1}}^{n+1-k}  \right\|    \bigg). 
\end{align*}
If $\mathcal{P} = \mathcal{P}_{K+1}$ and $u_{j} \in \mathcal{P}_{K+1} - \mathcal{P}_{K}$, then
\begin{align*}
\omega(u_{j-1}, u_{j}) = \omega(u_{j}, u_{j+1}) \leq \frac{1}{2^{K+1}} \omega(s,t). 
\end{align*}
Thus, for the first term in the above bracket, we have
\begin{align*}
\sum_{k=1}^{n} \left\|  R_{u_{j-1}, u_{j}}^{k} \otimes X_{u_{j}, u_{j+1}}^{n+1-k}  \right\|  &\leq \sum_{k=1}^{n} \epsilon \cdot \frac{\omega(u_{j-1}, u_{j})^{\frac{k - \delta}{p}}}{\beta (\frac{k}{p})!} \cdot \frac{\omega(u_{j}, u_{j+1})^{\frac{n+1-k}{p}}}{\beta (\frac{n+1-k}{p})!} \\
&\leq \frac{\epsilon}{\beta^{2}} \big[\frac{1}{2^{K+1}} \omega(s,t) \big]^{\frac{n+1-\delta}{p}} \sum_{k=1}^{n+1} \frac{1}{(\frac{k}{p})!(\frac{n+1-k}{p})!} \\
&\leq \frac{\epsilon p}{\beta^{2} (\frac{n+1}{p})!} \cdot 2^{\frac{\delta}{p}} \big(\frac{1}{2^{K}} \big)^{\frac{n+1-\delta}{p}} \omega(s,t)^{\frac{n+1-\delta}{p}}. 
\end{align*}
The same bound holds for the second term. For the third term, note that $\left\| R\right\|  \leq \left\| X \right\|  + \left\| Y \right\|$, thus twice of the previous bound works. By combining bounds for the three terms above, we have
\begin{align*}
\left\|  (\hat{X}_{s,t}^{\mathcal{P}_{K+1}} - \hat{Y}_{s,t}^{\mathcal{P}_{K+1}})^{n+1}  \right\| \leq  \left\|  (\hat{X}_{s,t}^{\mathcal{P}_{K+1}^{j}} - \hat{Y}_{s,t}^{\mathcal{P}_{K+1}^{j}})^{n+1}  \right\| + \frac{\epsilon p \cdot 2^{\frac{2p+\delta}{p}}}{\beta^{2} (\frac{n+1}{p})!}  \bigg[\frac{1}{2^{K}} \omega(s,t) \bigg]^{\frac{n+1-\delta}{p}}, 
\end{align*}
where $u_{j}$ is any point in $\mathcal{P}_{K+1} - \mathcal{P}_{K}$. By successively dropping the $2^{K}$ points in $\mathcal{P}_{K+1} - \mathcal{P}_{K}$, we have
\begin{align} \label{the first bound}
\left\|  (\hat{X}_{s,t}^{\mathcal{P}_{K+1}} - \hat{Y}_{s,t}^{\mathcal{P}_{K+1}})^{n+1}  \right\| \leq \left\|  (\hat{X}_{s,t}^{\mathcal{P}_{K}} - \hat{Y}_{s,t}^{\mathcal{P}_{K}})^{n+1}  \right\| + \frac{\epsilon p \cdot 2^{\frac{2p+\delta}{p}}}{\beta^{2} (\frac{n+1}{p})!} \big(\frac{1}{2^{K}} \big)^{\frac{n+1-p-\delta}{p}} \omega(s,t)^{\frac{n+1-\delta}{p}}. 
\end{align}
On the other hand, we have the bound
\begin{align*}
\left\|  (\hat{X}_{s,t}^{\mathcal{P}_{K+1}} - \hat{X}_{s,t}^{\mathcal{P}_{K+1}^{j}})^{n+1}   \right\| &\leq \sum_{k=1}^{n} \left\|  X_{u_{j-1}, u_{j}}^{k} \otimes X_{u_{j}, u_{j+1}}^{n+1-k}  \right\| \\
&\leq \frac{p}{\beta^{2} (\frac{n+1}{p})!} \big[  \frac{1}{2^{K}} \omega(s,t) \big]^{\frac{n+1}{p}}, 
\end{align*}
and the same bound holds for $Y$. Thus, we have
\begin{align*}
\left\| (\hat{X}_{s,t}^{\mathcal{P}_{K+1}} - \hat{Y}_{s,t}^{\mathcal{P}_{K+1}})^{n+1} \right\| \leq \left\| (\hat{X}_{s,t}^{\mathcal{P}_{K+1}^{j}} - \hat{Y}_{s,t}^{\mathcal{P}_{K+1}^{j}})^{n+1} \right\| + \frac{2p}{\beta^{2} (\frac{n+1}{p})!} \big[  \frac{1}{2^{K}} \omega(s,t) \big]^{\frac{n+1}{p}}. 
\end{align*}
Again, by successively dropping the $2^{K}$ points in $\mathcal{P}_{K+1} - \mathcal{P}_{K}$, we get
\begin{align} \label{the second bound}
\left\| (\hat{X}_{s,t}^{\mathcal{P}_{K+1}} - \hat{Y}_{s,t}^{\mathcal{P}_{K+1}})^{n+1} \right\| \leq \left\| (\hat{X}_{s,t}^{\mathcal{P}_{K}} - \hat{Y}_{s,t}^{\mathcal{P}_{K}})^{n+1} \right\| + \big( \frac{1}{2^{K}} \big)^{\frac{n+1}{p}-1} \frac{2p \omega(s,t)^{\frac{n+1}{p}}}{\beta^{2} (\frac{n+1}{p})!}. 
\end{align}
Combining \eqref{the first bound} and \eqref{the second bound}, we conclude the lemma. 

\end{proof}

\bigskip

\section{Proof of Theorem \ref{main theorem}}

\begin{proof}

We fix $s < t$. If $\epsilon \geq 2 \omega(s,t)^{\frac{\delta}{p}}$, then the assumption \eqref{condition for the control} on the $p$-variation of two paths automatically implies the theorem. So we may assume without loss of generality that $\epsilon < 2 \omega(s,t)^{\frac{\delta}{p}}$. In what follows, we let $N$ to be the unique integer such that
\begin{align} \label{the unique integer}
[\frac{1}{2^{N}} \omega(s,t)]^{\frac{\delta}{p}} \leq \frac{\epsilon}{2} < [\frac{1}{2^{N-1}} \omega(s,t)]^{\frac{\delta}{p}}. 
\end{align}
The idea is that, the integer $N$ above is the borderline where we switch from the 'uniform' distance assumption \eqref{uniform distance assumption} to the $p$-variation bound condition \eqref{condition for the control}. More precisely, since $X$ and $Y$ satisfy the assumptions of Lemma \ref{main lemma}, let $n + 1 = \lfl p \rfl + 1$, then $\forall K \leq N - 1$, we have
\begin{align} \label{first bound in the lemma}
\left\|  (\hat{X}_{s,t}^{\mathcal{P}_{K+1}} - \hat{Y}_{s,t}^{\mathcal{P}_{K+1}})^{n+1} \right\| - \left\| (\hat{X}_{s,t}^{\mathcal{P}_{K}} - \hat{Y}_{s,t}^{\mathcal{P}_{K}})^{n+1} \right\| \leq \frac{\epsilon p \cdot 2^{\frac{2p+\delta}{p}}}{\beta^{2} (\frac{n+1}{p})!}  \big(\frac{1}{2^{K}} \big)^{\frac{1-\{p\}-\delta}{p}} \omega(s,t)^{\frac{n+1-\delta}{p}}, 
\end{align}
and for all $K \geq N$, we have
\begin{align} \label{second bound in the lemma}
\left\|  (\hat{X}_{s,t}^{\mathcal{P}_{K+1}} - \hat{Y}_{s,t}^{\mathcal{P}_{K+1}})^{n+1} \right\| - \left\| (\hat{X}_{s,t}^{\mathcal{P}_{K}} - \hat{Y}_{s,t}^{\mathcal{P}_{K}})^{n+1} \right\| \leq \big( \frac{1}{2^{K}} \big)^{\frac{n+1}{p}-1} \frac{2p \omega(s,t)^{\frac{n+1}{p}}}{\beta^{2} (\frac{n+1}{p})!}. 
\end{align}
Now we proceed to prove the three cases in Theorem \ref{main theorem}.

\bigskip

\begin{flushleft}
\textbf{Case 1.} $\delta \in [0, 1 - \{p\})$. 
\end{flushleft}

\bigskip

We need to prove \eqref{first situation} for all integers $k$. In this situation, since the exponent for the level $\lfl p \rfl + 1$ is expected to be
\begin{align*}
\frac{\lfl p \rfl + 1 - \delta}{p} > 1, 
\end{align*}
then we can use the bound \eqref{first bound in the lemma} for all $K$, and we prove \eqref{first situation} directly by induction. Suppose \eqref{first situation} holds for $k = 1, 2, \cdots, n$, where $n \geq \lfl p \rfl$, then \eqref{first bound in the lemma} implies
\begin{align*}
\left\| (\hat{X}_{s,t}^{\mathcal{P}_{K+1}} - \hat{Y}_{s,t}^{\mathcal{P}_{K+1}})^{n+1} \right\| \leq \left\| (\hat{X}_{s,t}^{\mathcal{P}_{K}} - \hat{Y}_{s,t}^{\mathcal{P}_{K}})^{n+1} \right\| + \frac{\epsilon p \cdot 2^{\frac{2p+\delta}{p}}}{\beta^{2} (\frac{n+1}{p})!} \big( \frac{1}{2^{K}} \big)^{\frac{n+1-p-\delta}{p}} \omega(s,t)^{\frac{n+1-\delta}{p}}. 
\end{align*}
Also, by the construction of multiplicative functionals (Lemma \ref{construction of multiplicative extension}), we have
\begin{align*}
\left\| X_{s,t}^{n+1} - Y_{s,t}^{n+1}  \right\| &= \sum_{K=0}^{+\infty} \bigg( \left\| (\hat{X}_{s,t}^{\mathcal{P}_{K+1}} - \hat{Y}_{s,t}^{\mathcal{P}_{K+1}})^{n+1} \right\| -  \left\| (\hat{X}_{s,t}^{\mathcal{P}_{K}} - \hat{Y}_{s,t}^{\mathcal{P}_{K}})^{n+1} \right\|  \bigg) \\
&\leq \epsilon \cdot \frac{2^{\frac{2p+\delta}{p}} p}{\beta^{2} (\frac{n+1}{p})!} \omega(s,t) \sum_{K=0}^{+\infty} \big( \frac{1}{2^{K}} \big)^{\frac{n+1-p-\delta}{p}}\\
&\leq \epsilon \cdot \frac{\omega(s,t)^{\frac{n+1-\delta}{p}}}{\beta (\frac{n+1}{p})!}, 
\end{align*}
where the last inequality holds because $n+1-p-\delta > 0$ and $\beta$ satisfies the assumption \eqref{first condition}.

\bigskip

\begin{flushleft}
\textbf{Case 2.} $\delta = 1 - \{p\}$. 
\end{flushleft}

\bigskip

In this case, we need to prove \eqref{second situation} for all $k \geq \lfl p \rfl + 1$. We first prove it for level $\lfl p \rfl + 1$, and after that we can do induction. Note that by the definition of $\hat{X}$ and $\hat{Y}$ at the beginning of Section 3, we have
\begin{align*}
\left\| (\hat{X}_{s,t}^{\mathcal{P}_{N}} - \hat{Y}_{s,t}^{\mathcal{P}_{N}})^{n+1}  \right\| &= \sum_{K=0}^{N-1} \bigg( \left\| (\hat{X}_{s,t}^{\mathcal{P}_{K+1}} - \hat{Y}_{s,t}^{\mathcal{P}_{K+1}})^{n+1}  \right\|  -  \left\| (\hat{X}_{s,t}^{\mathcal{P}_{K}} - \hat{Y}_{s,t}^{\mathcal{P}_{K}})^{n+1}  \right\| \bigg). 
\end{align*}
Now we let $n + 1 = \lfl p \rfl + 1$. Note that $K \leq N-1$ in the above expression, so we can apply the bound \eqref{first bound in the lemma}, and get
\begin{align*}
\left\| (\hat{X}_{s,t}^{\mathcal{P}_{N}} - \hat{Y}_{s,t}^{\mathcal{P}_{N}})^{n+1}  \right\| \leq \frac{\epsilon p N}{\beta^{2} (\frac{n+1}{p})!} \cdot 2^{\frac{2p+1-\{p\}}{p}} \cdot \omega(s,t). 
\end{align*}
Since $2^{N-1} < 2^{\frac{p}{1-\{p\}}}\omega(s,t) / \epsilon^{\frac{p}{1-\{p\}}}$ by \eqref{the unique integer}, we have
\begin{align} \label{second situation comparison bound}
\left\| (\hat{X}_{s,t}^{\mathcal{P}_{N}} - \hat{Y}_{s,t}^{\mathcal{P}_{N}})^{n+1}  \right\| \leq p \cdot 2^{\frac{2p+1-\{p\}}{p}} \cdot \epsilon \bigg(1 + \frac{p}{1-\{p\}} + \log_{2}\frac{\omega(s,t)}{\epsilon^{(1-\{p\})/p}} \bigg) \frac{\omega(s,t)}{\beta^{2} (\frac{n+1}{p})!}. 
\end{align}
On the other hand, for $K \geq N$, the bound \eqref{second bound in the lemma} implies
\begin{align*}
\sum_{K=N}^{+\infty} \bigg( \left\|  (\hat{X}_{s,t}^{\mathcal{P}_{K+1}} - \hat{Y}_{s,t}^{\mathcal{P}_{K+1}})^{n+1} \right\|  -  \left\|  (\hat{X}_{s,t}^{\mathcal{P}_{K}} - \hat{Y}_{s,t}^{\mathcal{P}_{K}})^{n} \right\| \bigg) \leq \frac{2p \omega(s,t)^{\frac{n+1}{p}}}{\beta^{2} (\frac{n+1}{p})!} \sum_{K=N}^{+\infty}(\frac{1}{2^{K}})^{\frac{1 - \{p\}}{p}}. 
\end{align*}
Using $\frac{1}{2^{N}} \leq (\frac{\epsilon}{2})^{\frac{p}{1 - \{p\}}} / \omega(s,t)$ (by \eqref{the unique integer}), we have
\begin{align} \label{second situation natural bound}
\sum_{K=N}^{+\infty} \bigg( \left\|  (\hat{X}_{s,t}^{\mathcal{P}_{K+1}} - \hat{Y}_{s,t}^{\mathcal{P}_{K+1}})^{n} \right\|  -  \left\|  (\hat{X}_{s,t}^{\mathcal{P}_{K}} - \hat{Y}_{s,t}^{\mathcal{P}_{K}})^{n} \right\| \bigg) \leq \frac{2p \epsilon}{ 1 - (\frac{1}{2})^{\frac{1 - \{p\}}{p}}} \cdot \frac{\omega(s,t)}{\beta^{2} (\frac{n+1}{p})!}. 
\end{align}
Since $\beta$ satisfies the hypothesis \eqref{second condition}, by combining the two bounds \eqref{second situation comparison bound} and \eqref{second situation natural bound} above, we get
\begin{align*}
\left\| X_{s,t}^{n} - Y_{s,t}^{n}  \right\| &= \left\|  (\hat{X}_{s,t}^{\mathcal{P}_{N}} - \hat{Y}_{s,t}^{\mathcal{P}_{N}})^{n} \right\| + \sum_{K=N}^{+\infty} \bigg( \left\|  (\hat{X}_{s,t}^{\mathcal{P}_{K+1}} - \hat{Y}_{s,t}^{\mathcal{P}_{K+1}})^{n} \right\|  -  \left\|  (\hat{X}_{s,t}^{\mathcal{P}_{K}} - \hat{Y}_{s,t}^{\mathcal{P}_{K}})^{n} \right\| \bigg) \\
&< \epsilon \bigg(1 + \frac{p}{1-\{p\}} + \log_{2}\frac{\omega(s,t)}{\epsilon^{(1-\{p\})/p}} \bigg) \frac{\omega(s,t)}{\beta (\frac{n+1}{p})!}. 
\end{align*}
Replacing $\log_{2} \omega(s,t)$ by $\log_{2} \omega(0,1)$, we have proved \eqref{second situation} for level $\lfl p \rfl + 1$. The remaining can be proved by induction. Suppose \eqref{second situation} holds for $k = \lfl p \rfl + 1, \cdots, n$, then by breaking the sum into parts $1, \cdots \lfl p \rfl$ and $\lfl p \rfl + 1, \cdots n$, we have
\begin{align*}
\sum_{k=1}^{n} \left\| R_{u_{j-1}, u_{j}}^{k} \otimes X_{u_{j}, u_{j+1}}^{n+1-k}  \right\| < \frac{\epsilon p \cdot 2^{\frac{1 - \{p\}}{p}}}{\beta^{2} (\frac{n+1}{p})!} \bigg(1 + \frac{p}{1-\{p\}} + \log_{2}\frac{\omega(s,t)}{\epsilon^{(1-\{p\})/p}} \bigg) [\frac{1}{2^{K}} \omega(s,t)]^{\frac{n + \{p\}}{p}}, 
\end{align*}
and similar bounds hold for $\sum_{k=1}^{n} \left\| X_{u_{j-1}, u_{j}}^{k} \otimes R_{u_{j}, u_{j+1}}^{n+1-k} \right\|$ and $\sum_{k=1}^{n} \left\| R_{u_{j-1}, u_{j}}^{k} \otimes R_{u_{j}, u_{j+1}}^{n+1-k} \right\|$. Thus, same as before, by successively dropping the $2^{K}$ points in $\mathcal{P}_{K+1} - \mathcal{P}_{K}$, we get
\begin{align*}
&\phantom{111}\left\| (\hat{X}_{s,t}^{\mathcal{P}_{K+1}} - \hat{Y}_{s,t}^{\mathcal{P}_{K+1}} )^{n+1} \right\| - \left\| (\hat{X}_{s,t}^{\mathcal{P}_{K}} - \hat{Y}_{s,t}^{\mathcal{P}_{K}} )^{n+1} \right\| \\
&\leq \frac{4 \epsilon p \cdot 2^{\frac{1 - \{p\}}{p}}}{\beta^{2} (\frac{n+1}{p})!} \bigg(1 + \frac{p}{1-\{p\}} + \log_{2}\frac{\omega(0,1)}{\epsilon^{(1-\{p\})/p}} \bigg) \omega(s,t)^{\frac{n+ \{p\}}{p}} (\frac{1}{2^{K}})^{\frac{n - \lfl p \rfl}{p}}. 
\end{align*}
Since now the exponent $\frac{n - \lfl p \rfl}{p} > 0$, by summing over $K$ from $0$ to $+\infty$, we conclude that
\begin{align*}
\left\| X_{s,t}^{n+1} - Y_{s,t}^{n+1}  \right\| < \epsilon \bigg(1 + \frac{p}{1-\{p\}} + \log_{2}\frac{\omega(0,1)}{\epsilon^{(1-\{p\})/p}} \bigg) \frac{\omega(s,t)^{\frac{n + \{p\}}{p}}}{\beta (\frac{n+1}{p})!}, 
\end{align*}
as long as $\beta > 4p \cdot 2^{\frac{1 - \{p\}}{p}} / [1 - (\frac{1}{2})^{\frac{1}{p}}]$, which clearly satisfies \eqref{second condition}. Thus, we have completed the induction, and proved \eqref{second situation} for all $k \geq \lfl p \rfl$.

\bigskip

\begin{flushleft}
\textbf{Case 3.} $\delta \in (1 - \{p\}, 1]$. 
\end{flushleft}

This is possible only if $p$ is non-integer. We need to prove \eqref{third situation} for all $k \geq \lfl p \rfl + 1$. The proof is essentially the same to that for the case $\delta = 1 - \{p\}$. Similar as before, using the bound \eqref{first bound in the lemma} for $K \leq N-1$, we get
\begin{align*}
\left\| (\hat{X}_{s,t}^{\mathcal{P}_{N}} - \hat{Y}_{s,t}^{\mathcal{P}_{N}})^{n+1}  \right\| &\leq \frac{\epsilon p}{\beta^{2} (\frac{n+1}{p})!} \cdot 2^{\frac{2p+\delta}{p}} \cdot \omega(s,t)^{\frac{n + 1 - \delta}{p}} \sum_{K=0}^{N-1} 2^{\frac{K}{p}(\delta + \{p\} -1)} \\
&\leq \frac{\epsilon p}{\beta^{2} (\frac{n+1}{p})!} \cdot [2^{\frac{2p+\delta}{p}} / (2^{\frac{\delta - 1 + \{p\}}{p}} - 1)] \cdot \omega(s,t)^{\frac{n + 1 - \delta}{p}} \cdot 2^{\frac{N}{p}(\delta - 1 + \{p\})}. 
\end{align*}
Since $N$ satisfies \eqref{the unique integer} (we use the second inequality here), we have
\begin{align} \label{third situation comparison bound}
\left\| (\hat{X}_{s,t}^{\mathcal{P}_{N}} - \hat{Y}_{s,t}^{\mathcal{P}_{N}})^{n+1}  \right\| \leq \frac{p}{\beta^{2} (\frac{n+1}{p})!} \cdot [2^{\frac{3p+\delta}{p}} / 1 - (\frac{1}{2})^{\frac{\delta - 1 + \{p\}}{p}} ] \cdot \epsilon^{\frac{1 - \{p\}}{\delta}} \omega(s,t). 
\end{align}
On the other hand, similar to the previous case, using the bound \eqref{second bound in the lemma} for $K \geq N$, we have
\begin{align*}
\sum_{K=N}^{+\infty} \bigg( \left\|  (\hat{X}_{s,t}^{\mathcal{P}_{K+1}} - \hat{Y}_{s,t}^{\mathcal{P}_{K+1}})^{n+1} \right\|  -  \left\|  (\hat{X}_{s,t}^{\mathcal{P}_{K}} - \hat{Y}_{s,t}^{\mathcal{P}_{K}})^{n+1} \right\| \bigg) \leq \frac{2p \omega(s,t)^{\frac{n+1}{p}}}{\beta^{2} (\frac{n+1}{p})!} \sum_{K=N}^{+\infty}(\frac{1}{2^{K}})^{\frac{1 - \{p\}}{p}}. 
\end{align*}
Now, since $N$ satisfies the first inequality in \eqref{the unique integer}, we have
\begin{align} \label{third situation natural bound}
\sum_{K=N}^{+\infty} \bigg( \left\|  (\hat{X}_{s,t}^{\mathcal{P}_{K+1}} - \hat{Y}_{s,t}^{\mathcal{P}_{K+1}})^{n+1} \right\|  -  \left\|  (\hat{X}_{s,t}^{\mathcal{P}_{K}} - \hat{Y}_{s,t}^{\mathcal{P}_{K}})^{n+1} \right\| \bigg) \leq \frac{2p\epsilon^{\frac{1 - \{p\}}{\delta}}}{ 1 - (\frac{1}{2})^{\frac{1 - \{p\}}{p}}}  \cdot \frac{\omega(s,t)}{\beta^{2} (\frac{n+1}{p})!}. 
\end{align}
Since $\beta$ satisfies the hypothesis \eqref{third condition}, combining the bounds \eqref{third situation comparison bound} and \eqref{third situation natural bound}, we get
\begin{align*}
\left\| X_{s,t}^{n+1} - Y_{s,t}^{n+1}  \right\| < \epsilon^{\frac{1 - \{p\}}{\delta}} \cdot \frac{\omega(s,t)^{\frac{n+\{p\}}{p}}}{\beta (\frac{n+1}{p})!}, 
\end{align*}
where $n = \lfl p \rfl + 1$. For $k \geq \lfl p \rfl + 2$, we can apply the same induction procedure as in the previous case. Thus, we prove \eqref{third situation} for all $k \geq \lfl p \rfl + 1$.

\end{proof}

\bigskip

\section{An application}

We now explain briefly how our estimates apply to the problem mentioned in the introduction (see equations \eqref{the first equation} and \eqref{the second equation} and assumption \eqref{uniformly close}). Since $A$ is a linear map, the solutions can be written as
\begin{align*}
x_{t} = \sum_{n=0}^{+\infty} A^{*n}x_{0} \int_{0 < u_{1} < \cdots < u_{n} < t} d\gamma_{u_{1}} \cdots d\gamma_{u_{t}}, 
\end{align*}
and
\begin{align*}
y_{t} = \sum_{n=0}^{+\infty} A^{*n}y_{0} \int_{0 < u_{1} < \cdots < u_{n} < t} d\tilde{\gamma}_{u_{1}} \cdots d\tilde{\gamma}_{u_{t}}, 
\end{align*}
Let $X$ and $Y$ denote the signatures of $\gamma$ and $\tilde{\gamma}$, then Theorem \ref{main theorem} implies that
\begin{align} \label{signatures are close}
\left\| X_{s,t}^{n} - Y_{s,t}^{n}  \right\| < \epsilon (1 + \log_{2} \frac{C}{\epsilon}) \cdot \frac{\omega(s,t)^{n-1}}{\beta n!}, \qquad \forall n \geq 2, 
\end{align}
where $C \leq \omega(0,1)$ is a generic constant. Let $x_{s,t} = x_{t} - x_{s}$ and $y_{s,t} = y_{t} - y_{s}$, then we have
\begin{align*}
|x_{s,t} - y_{s,t}| &\leq \sum_{n=0}^{+\infty} \left\| A \right\|^{n} \left\| X_{s,t}^{n} - Y_{s,t}^{n} \right\| \\
&\leq \left\| A \right\| \cdot \left\| X_{s,t}^{1} - Y_{s,t}^{1}  \right\| + \sum_{n=2}^{+\infty} \left\| A \right\|^{n} \epsilon (1 + \log_{2} \frac{C}{\epsilon}) \cdot \frac{\omega(s,t)^{n-1}}{\beta (n-1)!} \\
&\leq 2 \left\| A \right\| \min\{\epsilon, \omega(s,t)\} + \epsilon (1 + \log_{2}\frac{C}{\epsilon}) \frac{\left| A \right\|}{\beta}  (e^{\left| A \right\| \omega(s,t)} - 1). 
\end{align*}
In particular, the difference of the two solutions $|x_{t} - y_{t}|$ are of order $\epsilon$ up to a logarithmic correction, uniformly in $t \in [0,1]$.

\bigskip

\end{document}